\documentclass[a4paper,12pt]{article}
\usepackage{amssymb, enumerate, xspace, graphicx}
\usepackage[all]{xy}
\SelectTips{cm}{}

\usepackage[french,english]{babel}
\usepackage[latin1]{inputenc}
\usepackage[T1]{fontenc}
\usepackage{aeguill}
\usepackage{epsfig}
\usepackage{amsmath}
\usepackage{amsfonts}
\usepackage{amssymb}
\usepackage{amsthm}
\usepackage{amscd}
\usepackage{amsxtra}
\usepackage{upref}
\usepackage{multicol}
\usepackage{graphicx}
\usepackage{graphicx}
\usepackage[left=3cm,bottom=2.5cm,right=2.5cm,top=3cm]{geometry}
\usepackage{color}
\usepackage{float}
\usepackage{multirow}
\usepackage{dcolumn}

\numberwithin{equation}{section}

1

\numberwithin{subsection}{section}

\allowdisplaybreaks[1]

\newenvironment{enumeratei}
{\begin{enumerate}[\upshape (i)]}
{\end{enumerate}}

\newtheorem*{namedtheorem}{\theoremname}
\newcommand{\theoremname}{testing}

\newtheorem{theorem}{Theorem}[section]
\newtheorem{proposition}[theorem]{Proposition}
\newtheorem{proposition-definition}[theorem]
{Proposition-Definition}
\newtheorem{corollary}[theorem]{Corollary}
\newtheorem{lemma}[theorem]{Lemma}

\theoremstyle{definition}

\newtheorem{remark}[theorem]{Remark}

\theoremstyle{remark}

 \newcommand\cH{\mathcal{H}}
\newcommand\cI{\mathcal{I}} 
 \newcommand\cL{\mathcal{L}}
\newcommand\cM{\mathcal{M}} 
\newcommand\cO{\mathcal{O}} 
 
\newcommand\cS{\mathcal{S}} 
\newcommand\cU{\mathcal{U}}

 \newcommand\PP{\mathbb{P}}

 \newcommand\ZZ{\mathbb{Z}}

\newcommand\OO{\mathcal{O}_C}

\newcommand\arr{\ifinner\to\else\longrightarrow\fi}

\newcommand\arrto{\ifinner\mapsto\else\longmapsto\fi}

\newcommand\eqdef{\overset{\mathrm{\scriptscriptstyle def}} =}

\def\displaytimes_#1{\mathrel{\mathop{\times}\limits_{#1}}}

\def\displayotimes_#1{\mathrel{\mathop{\bigotimes}\limits_{#1}}}

\newcommand\ext{\operatorname{Ext}}

\newcommand\pic{\operatorname{Pic}}

\newcommand\Jac{\operatorname{Jac}}

\newcommand\Kum{\operatorname{Kum}}

\newdir{ >}{{}*!/-5pt/@{>}}

\newcommand\doublelong[2]{\mathbin{\xymatrix{{}\ar@<3pt>[r]^{#1}
\ar@<-3pt>[r]_{#2}&}}}

\newlength{\ignora}

\newcommand{\Sec}{\mathrm{Sec}}

\newcommand{\ki}{\mathcal{I}}

\newcommand{\sym}{\operatorname{Sym}}

\newcommand{\lra}{\longrightarrow}

\newcommand{\ra}{\rightarrow}

\newcommand{\Pd}{\PP^{3g-2}_D}

\newcommand{\su}{\mathcal{SU}_C(2)}

\begin{document}


\begin{center}
\LARGE{On rational maps between moduli spaces of curves and of vector bundles}\\
\small{A structure theorem for $\su$ and the moduli of pointed rational curves}
\end{center}

\vspace{2pt}

\begin{center}
\large{\textsc{Alberto Alzati, Michele Bolognesi}}
\end{center}




\begin{abstract}

\vspace{2pt} \footnotesize{ \noindent Let $\su$ be the moduli space of rank 2 semistable vector bundles with trivial determinant on a smooth complex algebraic curve $C$ of genus $g>1$, we assume $C$ non-hyperellptic if $g>2$. In this paper we construct large families of pointed rational normal curves over certain linear sections of $\su$. This allows us to give an interpretation of these subvarieties of $\su$ in terms of the moduli space of curves $\cM_{0,2g}$. In fact, there exists a natural linear map $\su \ra \PP^g$ with modular meaning, whose fibers are birational to $\cM_{0,2g}$, the moduli space of $2g$-pointed genus zero curves. If $g<4$, these modular fibers are even isomorphic to the GIT compactification $\cM_{0,2g}^{GIT}$.
The families of pointed rational normal curves are recovered as the fibers of the maps that classify extensions of line bundles associated to some effective divisors.}
\end{abstract}


\section{Introduction}

The first ideas about moduli of vector bundles on curves date back some eighty years, when in \cite{weilAB} for the first time the author suggested the idea that an analogue of Picard varieties could be provided by higher rank bundles. Then, in the second half of the last century, a more complete construction of these moduli spaces was carried out, mainly by Mumford, Newstead
\cite{MumNews} and the mathematicians of the Tata institute, e.g. \cite{ramanara}. Let us denote by $\cS\cU_C(r)$ the moduli space of semistable vector bundles of rank $r$ and trivial determinant on a smooth, complex curve $C$ of genus $g$. If $g\neq 2$ we will also assume throughout the paper that $C$ is not hyperelliptic.

Some spectacular results have been obtained on the projective structure of these moduli spaces in low genus and rank, especially thanks to the relation with the work on theta functions and classical algebraic geometry of A.B.Coble \cite{cob:agtf}. This interplay has produced a flourishing of beautiful results (see \cite{paulydual}, \cite{bove:cob}, \cite{orte:cob}, \cite{Quang}, or \cite{do:pstf} for a survey)  where both classical algebraic geometry and modern moduli theory come into play.

\smallskip

On the other hand, the theory of moduli spaces of curves is a cornerstone of modern algebraic geometry. Its importance lies not only in its own advances, but also in the impulse it has had in developing new technical tools that have remarkably improved the background of the modern algebraic geometer. Let us just mention as an example \cite{dm69}, where the notion of algebraic stack was first introduced.

\smallskip



The interaction between these two moduli theories has always proven to be fruitful and sometimes revolutionary. Among its highlights are without any doubt the results of Laszlo and Sorger on conformal blocks (see \cite{ivoecristo} and \cite{sorgerconformal}) or the beautiful paper by Kouvidakis on the Picard group of the relative moduli space of vector bundles over $\cM_g$ \cite{kuvi1}.

\smallskip

This paper explores further this interplay, by drawing a new link
between the two moduli theories. In fact, we describe a natural way
to construct large universal families of pointed rational curves
over some subvarieties of $\su$. We consider the morphism (where
$^*$ means dual vector spaces):

$$\su \ra \PP H^0(\su, \cL)^* = |\cL|^*$$

\noindent defined by the determinant bundle $\cL$. By identifying
this target projective space with the projectivized complete linear
system $|2\Theta|$ on the Picard variety $\pic^{g-1}(C)$, we are able to define a subspace $\PP_c \subset
|2\Theta|$ depending on the choice of an effective divisor $D$ of
degree $g$ on $C$ (see Section \ref{cuore}). By composing the
morphism $\su \ra |2\Theta|$ with the projection with center
$\PP_c$, we obtain a rational map $p_{\PP_c}:\su \dashrightarrow
|2D|\cong \PP^g$. Our main result then shows that this rational map
has a natural modular interpretation, as the generic fiber is
birational to $\cM_{0,2g}$.

\begin{theorem}\label{global}
Let $C$ be a smooth complex curve of genus $g>1$, non-hyperelliptic if $g>2$, and let $D$ be a fixed effective degree $g$ divisor on $C$. Then:

\begin{enumeratei}
\item There exists a rational fibration $p_{\PP_c}:\su \dashrightarrow
|2D|\cong \PP^g$ whose general fiber is birational to the moduli space $\cM_{0,2g}$ of $2g$-pointed genus zero curves.

\item The generic fiber $p_{\PP_c}^{-1}(N)$, for $N\in |2D|$ is dominated by a $2g$-pointed projective space $\PP^{2g-2}$; the fibers of the map are rational normal curves passing through the $2g$ fixed points. This family of rational curves induces the birational moduli map of point \textrm{(i)} via the universal property of $\cM_{0,2g}$.

\item There exists a birational inverse $\cM_{0,2g} \dashrightarrow p_{\PP_c}^{-1}(N)\subset \su$. The required vector bundle in $p_{\PP_c}^{-1}(N)$
is obtained as a twist by $\OO(D)$ of the kernel of a surjective sheaf morphism 
$\OO^{\oplus 2} \ra \mathcal{O}_N$ associated to a point configuration $(p_1,\dots,p_{2g})\in \cM_{0,2g}$, with $p_i\in \PP^1$. 
\end{enumeratei} 
\end{theorem}


For $g=2,3$ we have a more precise statement (see Proposition \ref{segre}).

\begin{theorem}
If $g=2,3$ then the general fiber of the fibration is isomorphic to the GIT compactification $\cM_{0,2g}^{GIT}$ of the moduli space of $2g$-pointed rational curves.
\end{theorem}

On the negative side, we show that if $g>3$ such an isomorphism is not possible (see Remark \ref{negative}).\\

The main tools of our proofs, apart from standard descriptions of the geometry of the Jacobian and of the theta series, are related to the results about the maps classifying extension classes described, in particular, in \cite{ab:rk2} and \cite{ln85}. In fact the rational normal curves that appear in Theorem \ref{global} parametrize certain extension classes and are the fibers of a classifying map.
In this paper the classifying maps under consideration will be the forgetful rational maps whose domain is $\PP \ext^1(L,L^{-1})$, for some line bundle $L$. These maps send an extension equivalence class

$$ 0 \ra L^{-1} \ra E \ra L \ra 0$$

\noindent to the corresponding bundle $E\in \su$, which in fact has clearly trivial determinant. Quite often these maps are not defined on all the projective space $\PP \ext^1(L,L^{-1})$ since this contains also non-semistable extensions. These non-semistable extension classes correspond to the points of certain varieties of secants of the projective model of the curve $C$ contained in $\PP \ext^1(L,L^{-1})=|K + 2L|^*$ (see Sect. 2.2 for details).\\

Let $U\subset \su$ be an open subset of a general fiber of $p_{\PP_c}$. The idea of our proof is to consider the closure of the fibers over $U$ of one of the classifying maps described above, namely

$$\varphi_D: \PP \ext^1(\cO_C(D),\cO_C(-D)) \dashrightarrow \su$$

\noindent for a general reduced effective degree $g$ divisor $D$, and show that they are a flat family of pointed rational curves over $U$. The existence of this family, by the universal property of $\cM_{0,2g}$, induces a birational map between the general fiber of $p_{\PP_c}$ and $\cM_{0,2g}$ (see Sections 5,6 and 7). A very naif picture is given in Figure \ref{naive}, in the last Section.

This is done via Kapranov's construction \cite{kapra} of the
Mumford-Knudsen compactification $\overline{\cM}_{0,n}$ as a blow-up
of the projective space $\PP^{n-3}$ and by considering the relation
between $\overline{\cM}_{0,n}$ and rational normal curves in
$\PP^{n-2}$ passing through $n$ fixed general points, as explored in
\cite{keelkernan}. In fact, recall that there exists a projective
model of $C\subset \PP \ext^1(\cO_C(D),\cO_C(-D))$ and let $U$ be a suitable open subset of the generic fiber of
$p_{\PP_c}$. Then there exist a $(2g-2)$-dimensional projective space
$\PP^{2g-2}\subset \PP \ext^1(\cO_C(D),\cO_C(-D))$, $2g$-secant to
$C$, such that the general rational normal curve in $\PP^{2g-2}$ passing through the
$2g$ secant points is contracted by $\varphi_D$ onto a point of $U$. This gives
the desired family over $U$. Finally, we also display a birational inverse map from $\cM_{0,2g}$ to $\su$, depending on the choice of a reduced divisor $N\in |2D|$. This consists in associating a surjective sheaf morphism 

\begin{equation}\label{display}
\OO^{\oplus 2} \ra \mathcal{O}_N
\end{equation}

to a configuration of points $(p_1,\dots,p_n)$ in $\PP^1$. The kernel of morphism (\ref{display}) is a rank 2 vector bundle and the inverse moduli map sends the configuration of points to the twist by $\OO(D)$ of this kernel bundle.\\


A final warning: we have always tried to state explicitly when it is important to consider the S-equivalence class of a decomposable vector bundle and not just the bundle. When it is not stated it is usually not relevant. We often also say $s$-class for S-equivalence class. The symbol $\equiv$ will mean linear equivalence of divisors and by $< S >$ we will denote the linear span of a subset $S$ of a projective space. 

\textit{Acknowledgment:} We would like to thank the referee for many useful comments and corrections that have helped us to improve in a significant way this paper. A special thank goes to Igor Dolgachev and Christian Pauly for inspiring conversations on the subjects of this paper.


\medskip

\textbf{Description of contents.}
In Section 2 we give a brief account of the $2\Theta$ linear series on the Picard variety $\pic^{g-1}(C)$ and its relation with $\su$ plus a description of the map $\varphi_D: \PP \ext^1(\cO_C(D),\cO_C(-D)) \dashrightarrow \su$ classifying extension classes, for an effective divisor $D$ of degree $g$.\\

\smallskip
   Section 3 deals mainly with the fibers of exceptional dimension of the preceding map and with another analogous map $\varphi_B$ that classifies extension classes in\\ $\PP \ext^1(\cO_C(B),\cO_C(-B))$, where $B\eqdef K-D$ is the "Serre dual" divisor of $D$. The image $Im(\varphi_B)$ of the corresponding classifying map $\varphi_B$ is also described. \\
\smallskip

In Section 4 we prove the core theorem of the paper, which describes in terms of vector bundles the restriction to $\su$ of the projection in $|2\Theta|$ with center $\langle Im(\varphi_B) \rangle$. It turns out to be the map that associates to $E\in \su$ the determinant of the two sections of $E\otimes\cO_C(D)$. The rest of the paper is devoted to describing the maps between the fibers of this projection and the moduli spaces $\cM_{0,2g}$.\\
\smallskip

This is done by observing (Section 5) that the generic fibers of $\varphi_D$ are rational normal curves passing through $2g$ fixed general points in $\PP^{2g-2}$ and using the results of \cite{kapra} that outline the classical bijection between these curves and the points of $\cM_{0,2g}$. Moreover the $g=2$ case is recalled from \cite{bol:kumwed} and interpreted in the light of our new results.\\
\smallskip

In Section 6 we describe in detail the case of genus $3$ by showing the birationality between the Coble quartic and a fibration in Segre cubics and finally in Section 7 we go through the relation of $\su$ with $\cM_{0,2g}$ for $g\geq 4$, completing the proof of Theorem \ref{global}.\\

\section{Moduli of Vector Bundles and Classification of Extensions}

\subsection{Vector bundles and theta linear systems}

Let $C$ be a smooth genus $g\geq 2$ algebraic curve, non-hyperelliptic if $g>2$. Let $\pic^d(C)$ be the Picard variety that parametrizes linear equivalence classes of all degree $d$ line bundles on $C$. $\pic^0(C)$ will be often denoted by $\Jac(C)$. We recall that there exists a canonical divisor $\Theta\subset \pic^{g-1}(C)$ that set theoretically is defined as

$$\Theta\eqdef\{L\in \pic^{g-1}(C)| h^0(C,L)\neq 0 \}.$$

Let moreover $\su$ be the moduli space of semi-stable rank 2 vector bundles on $C$ with trivial determinant. More precisely, $\su$ does not parametrize isomorphism classes of vector bundles on $C$, but S-equivalence classes. We recall this equivalence for the reader's covenience. It is well-known that every semistable vector bundle $E$ admits a \textit{Jordan-Holder filtration}

$$0 = E_0 \subsetneq E_1 \subsetneq \cdots \subsetneq E_{k-1} \subsetneq E_k = E,$$

\noindent such that each successive quotient $E_i/E_{i-1}$ is stable of slope equal to $\mu(E)$, for $i = 1, \dots , k.$ We call

$$gr(E) = \bigoplus_{i=1}^k E_i/E_{i-1}$$

\noindent the \textit{graded bundle} associated to $E$. Two semi-stable vector bundles $E$ and $E'$ on $C$ are said to be S-equivalent if $gr(E)\cong
gr(E')$. In particular, two stable bundles $E$ and $E'$ are S-equivalent if and only if they are isomorphic. \newline

It is well known that $\su$ is locally factorial and that $\pic (\su) = \ZZ$ \cite{dn:pfv}, generated by a line bundle
$\cL$ called the \textit{determinant bundle}. On the other hand, for $E\in \su$ let us define

$$\theta(E)\eqdef\{L\in \pic^{g-1}| h^0(C,E\otimes L)\neq 0\}.$$

While in higher rank there are examples of vector bundles $F$ such that $\theta(F)$ is the whole $\pic^{g-1}(C)$, if the rank is two then $\theta(E)$ is a divisor in the linear system of $2\Theta$ for every $E\in \su$. This gives the celebrated theta map

$$\theta:\su  \longrightarrow  |2\Theta|= \PP^{2^g-1}.$$

We recall that the linear system $|2\Theta|$ on $\pic^{g-1}(C)$ contains the Kummer variety $\Kum(C)$ of $C$. This is the quotient of the Jacobian of $C$ by the involution $x \mapsto -x$ and the map

\begin{eqnarray*}
k : \Jac(C) & \longrightarrow & |2\Theta|,\\
x & \mapsto & \Theta_x + \Theta_{-x},
\end{eqnarray*}

\noindent factors through an embedding $\kappa: \Kum(C) \hookrightarrow |2\Theta|$. From now on $\Kum(C)$ will be considered as a subvariety of $|2\Theta|$. The geometry of the Kummer variety is intricately related to the geometry of $\su$, in fact $\Kum(C)$ coincides exactly with the non-stable part of $\su \subset |2\Theta|$, which in fact consists of bundles of the form $L\oplus L^{-1}$, with $L \in \Jac(C)$.\\
One of the most striking properties of $\theta$ is the following.

\begin{theorem}\cite{bo:fib1}\label{wirti}
There is a canonical isomorphism

$$H^0(\su,\cL) \cong H^0(\pic^{g-1}(C), 2 \Theta)^*,$$

\noindent making the following diagram commutative

\ \ \ \ \ \ \ \ $\xymatrix{ & & |\cL|^* \ar[dd] \\
\Kum(C)\ar@{^{(}->}[r] \ar@/_/[drr]^{\kappa} & \su \ar[ur] \ar[dr]^{\theta} \\
 & & |2\Theta|  }$

\end{theorem}

Furthermore, thanks to \cite{BVtheta} and \cite{vgiz} it is known that $\theta$ is an embedding for every $g\geq 2$.
  \\

The lower genus $\su$ moduli spaces deserve a special mention for their significant and beautiful geometry.

\begin{itemize}
\item If $g=2$ then $\su\cong |2\Theta|= \PP^3$ and its semi-stable boundary is the well known Kummer quartic surface. \cite{rana:cra}
\item If $g=3$ then Ramanan and Narashiman \cite{rana:cra} showed that $\su$ is a quartic hypersurface in $\PP^7\cong|2\Theta|$ which is singular along $\Kum(C)$. Several years earlier Coble \cite{cob:agtf} had showed the existence of a unique such quartic hypersurface, that nowadays is named after him. Remarkably, this variety is also self-dual \cite{paulydual}.

\end{itemize}

\subsection{The classifying maps}\label{sectclass}

Let $D$ be a general degree $g$ effective divisor on $C$ and let us introduce the $(3g-2)$-dimensional projective space

$$\PP^{3g-2}_D\eqdef\PP \ext^1(\OO(D),\OO(-D))=|K+2D|^*.$$

A point $e\in \Pd$ corresponds to an isomorphism class of extensions

$$0\lra \OO(-D) \stackrel{i_e}{\lra} E_e \stackrel{\pi_e}{\lra} \OO(D) \lra 0.\ \ \ \ \ (e)$$

There is a natural rational, surjective, forgetful map $\varphi_D$ that is defined in the following way:

\begin{eqnarray*}
\varphi_D: \Pd & \dashrightarrow & \su, \\
e & \mapsto & E_e.
\end{eqnarray*}

This kind of classifying map has been described by Bertram in \cite{ab:rk2}. In our case Theorem 2 of Bertram's paper gives an isomorphism

\begin{equation}\label{bertro}
H^0(\su,\cL) \cong H^0(\Pd,\cI_C^{g-1}(g)),
\end{equation}

\noindent where $\cI_C$ is the ideal sheaf of the degree $4g-2$ curve $C\subset \Pd$ embedded by $|K+2D|$. This means that the classifying map is given by the full linear system of forms of degree $g$ that vanish with multiplicity at least $g-1$ on $C$.

Throughout the paper, we will make massive use of the following Proposition, that comes directly from \cite{ln85} Prop. 1.1 (see also \cite{oxpau:heis} Remark 8.2).

\begin{proposition}\label{span}
Let us fix an effective divisor $G$ on $C$, $(e)$ any extension class in $\Pd$ and $E_e$ the corresponding rank 2 vector bundle. We denote by $<G>$ the linear span of $G$ in $\Pd$. Then, $e\in <G>$ if and only if there is a sheaf injection $\cO_C(D-G) \stackrel{i_G}{\hookrightarrow} E_e$ such that $G$ is the zero divisor of $\pi_e \circ i_G$.  We will refer to such an extension class by saying that $(e)$ has a sheaf injection \textit{vanishing on} $G$ for.
\end{proposition}

In the following we will denote by $\Sec^n(C)$ the variety of $n$-secant $(n-1)$-planes.


\noindent The following Lemma is an easy consequence of Prop. \ref{span}.

\begin{lemma}\label{linear}
Let $(e)$ be an extension class in $\Pd$, then the vector bundle $E_e$ is not semi-stable if and only if $e\in \Sec^{g-1}(C)$ and it is not stable if $e \in \Sec^{g}(C)$.
\end{lemma}

One can draw some conclusions from Lemma \ref{linear}. The first one
is the following.

\begin{corollary}
The image of the secant variety $\Sec^g(C)$ via the classifying map $\varphi_D$ is the Kummer variety $\Kum(C)$.
\end{corollary}



On the other hand the linear system $|\mathcal{I}_C^{g-1}(g)|$ can be described in terms of a secant variety of $C$.

\begin{lemma}\label{luoghi}
The two linear systems $|\mathcal{I}_C^{g-1}(g)|$ and $|\ki_{\Sec^{g-1}(C)}(g)|$ on $\Pd$ coincide.
\end{lemma}

\begin{proof}
In order to show the claim it is useful to consider the elements of the two linear systems as symmetric $g$-linear forms on the vector space $H^0(C,K+2D)^*$. Let $F$ be such a form. Then, $F$ belongs to $|\ki^{g-1}_C(g)|$ if and only if 
\begin{equation}\label{conditio}
F(p_1,\dots, p_g)=0,\ \mathrm{for\ any\ }p_i\in C\ \mathrm{and\  }p_k=p_j\ \mathrm{for\ some\ }1\leq k,j \leq g.
\end{equation}

On the other hand, a form $G$ belongs to $|\ki_{\Sec^{g-1}(C)}(g)|$ if and only if $G(p,p,\dots,p)=0$ for any point $p\in \Pd$ that can be written as $p=\sum_{i=1}^{g-1}\lambda_i p_i$, for any scalars $\lambda_i$, and $p_i\in C$, $i=1,\dots, g-1$. Now, let us develop $G$ as a polynomial in $\lambda_i$. Appropriate choices of $\lambda_i$ show that $G$ verfies exactly condition (\ref{conditio}). On the other hand, let us develop in the same way $F(p,p,\dots,p)$. It is not hard to see that $F$ vanishes identically for any choice of the $\lambda_i$, by using condition (\ref{conditio}).
\end{proof}


\smallskip





In the next section we will describe the fibers of dimension bigger
than expected over stable bundles. Now we give a description of
the fibers of $\varphi_D$ over the points of $\Kum(C)$. This is completely
different from that of the fibers over the generic stable bundle,
and rather technical. The next proposition is not necessary in order to understand the main results of this
paper, but we included it for sake of completeness. The impatient
reader can skip it, and he can do the same with Section
\ref{pencil}.

\begin{proposition}\label{piani}
Let $L\in \Jac(C)$ be a general degree 0 line bundle and
let $\Jac(C)[2]$ be the group of 2-torsion points of $\Jac(C)$. Then the fiber of $\varphi_D$ over the S-equivalence class of $L\oplus L^{-1}$ is given by two $\PP^{g-1}\subset \Pd$. If $L\in \Jac(C)[2]$ the two $\PP^{g-1}$ coincide.
\end{proposition}

\begin{proof}
Let $L\oplus L^{-1}$ be a general point of $\Kum(C)$.

\smallskip

The Abel-Jacobi map

\begin{eqnarray}\label{abeljaco}
\sigma_g:\sym^g (C) & \lra & \Jac(C),\\
p_1+\cdots+p_g & \mapsto & \OO(D - p_1 - \cdots -p_g)
\end{eqnarray}

\noindent is surjective and generically one-to-one. In this proof we will consider only decomposable bundles that are contained in the image in $\Kum(C)$ of the one-to-one locus. The locus where the fiber of $\sigma_g$ is positive dimensional has codimension 2 in $\Kum(C)$ and it will be described later, in Proposition \ref{exkum}. For a general $L$, there exists only one effective divisor $x_1+ \cdots +x_g$ such that $L\equiv \cO_C(D-x_1+ \cdots +x_g)$. Similarly, there exists a unique effective divisor $y_1+\dots + y_g$ (linearly equivalent to $2D - x_1 -\dots - x_g$) such that $L^{-1}\equiv \cO_C(D - y_1 \cdots - y_g).$\\

Now let us consider the two projective spaces $\PP^{g-1}\subset \Pd$ spanned respectively by $\sum_{i=1}^gx_i$ and $\sum_{i=1}^gy_i$. By Prop. \ref{span}, the image via $\varphi_D$ of any extension class $(e)$ contained in these two $(g-1)$-planes is exactly the $s$-class of $L \oplus L^{-1}$. In fact, let us call $E_e$ the rank two bundle in the middle of the exact sequence defining $(e)$. If $(e)$ is contained in one of the two projective spans, then $E_e$ contains either $L$ or $L^{-1}$ as a subsheaf. Hence $E_e$ is S-equivalent to $L\oplus L^{-1}$, since it has trivial determinant.

\smallskip

Let us show now the opposite inclusion. Let us consider a vector bundle $E$ in the S-equivalence class of $L \oplus L^{-1}$. Then either $L$ or $L^{-1}$ is a subsheaf of $E$. Suppose that $L$ injects in $E$ and let us consider an extension $(e)$ such that $\varphi_D(e)=E$. Then the composed map $L \hookrightarrow E \stackrel{\pi_e}{\lra} \cO_C(D)$ has zero divisor equal to $\sum_i x_i$ and, by Prop. \ref{span}, we have that $(e)\in <x_1+\dots + x_g>$. If $L\not\subset E$ then $L^{-1}\subset E$ and $(e)\in <y_1+\dots + y_g>$.

Moreover, if $L\in \Jac(C)[2]$ then $L\cong L^{-1}$ and the two effective divisors just described coincide, hence in this case the fiber is a double $\PP^{g-1}$.
\end{proof}

In the next section we will describe the exceptional fibers over stable bundles.

\section{A dual classifying map and the exceptional fibers}\label{excess}

\subsection{Exceptional fibers of $\varphi_D$ over stable bundles}

We recall that $\dim(\su)=3g-3$, hence the generic fiber of $\varphi_D$ has dimension one. In fact for a general stable bundle $E\in \su$, we have dim$(\PP H^0(C,E\otimes\OO(D)))=1$  and the generic point of $\PP(H^0(C,E\otimes\OO(D)))$ defines in fact an extension in $\Pd$ (some of the points of $\PP(H^0(C,E\otimes\OO(D)))$ are just sheaf injections). 
However there is a proper subset of $\su$ made up of stable bundles for which $h^0(C,E\otimes \OO(D))>2$ and thus dim$(\varphi_D^{-1}(E))>1$. In order to describe these particular bundles we introduce the "Serre dual" divisor

\begin{equation}\label{defB}
B\eqdef K - D.
\end{equation}

\noindent Note that $deg(B)=g-2$. Of course, we can define a rational map

$$\varphi_B:\PP \ext^1(\OO(B),\OO(-B))\arr \su$$

\noindent analogous to $\varphi_D$ that classifies the extensions of the following type:

$$0 \arr \OO(-B) \arr E \arr \OO(B) \arr 0.$$

\newcommand\Pb{\PP_B^{3g-6}}

We remark that also in this case we have

\begin{equation}\label{defPB}
\PP_B^{3g-6}\eqdef\PP \ext^1(\OO(B),\OO(-B))=|K+2B|^*.
\end{equation}

Let us denote by $\overline{\varphi_B(\PP_B^{3g-6})}$ the closure of the image in $\su$.

\begin{proposition}
Let $E\in \su$ be a stable bundle, then $\dim(\varphi_D^{-1}(E))\geq 2$ if and only if $E$ is contained in the image of $\varphi_B$.
\end{proposition}

\begin{proof}
Let $E$ be a stable bundle, then by Riemann-Roch and Serre duality Theorems we have that

$$h^0(C,E\otimes\OO(D))= h^0(C,E\otimes \OO(B)) +2g -2(g-1),$$

\noindent which implies that $h^0(C,E\otimes\OO(D))>2$ if and only if there exists a map\\ $\OO(-B)\arr E$. In turn this means that $E$ is in the closure of the image of  $\varphi_B$. In fact, the vector bundles contained in $\overline{\varphi_B(\PP_B^{3g-6})}$ are exactly those that admit a non-zero sheaf morphism $\cO_C(-B)\ra E$.
\end{proof}

 If $g>2$, there exists a projective model of the curve $C\subset \Pb$ of degree $4g-6$ embedded by $|K+2B|^*$ and, by Prop. \ref{span}, $\Sec^{g-3}(C)$ is the locus of non-semistable extensions in $\Pb$. Note that, of course, if $g=2$ then $\Pb$ is just a point and there is no curve contained therein. By a conjecture of Oxbury and Pauly (\cite{oxpau:heis}, Conj. 10.3), subsequently proved by Pareschi and Popa (\cite{papo}, Thm. 4.1), the map $\varphi_B$ is given by the complete linear system $|\cI_C^{g-3}(g-2)|$ on $\PP_B^{3g-6}$. This linear system (see the proof of the same conjecture) has projective dimension $\left[\sum_{i=0}^{g-2}\left(\begin{array}{c}
g \\ i\\
\end{array}\right)\right]-1$ and it is identified with a linear subspace of $|2\Theta|$. Moreover the image in $|\cI_C^{g-3}(g-2)|$ of the open semistable locus of $\Pb$ is non degenerate and by definition it is contained in $\su$.

\subsection{Exceptional fibers of $\varphi_D$ over non-stable bundles}\label{pencil}

In Prop. \ref{piani} we described the fibers of $\varphi_D$ over general points
of $\Kum(C)$. We will see now that the fibers of $\varphi_D$ over
$\Kum(C)$ that have exceptional dimension are exactly those over
$\Kum(C)\cap \overline{\varphi_B(\PP_B^{3g-6})}$. These fibers have
dimension bigger than those over the generic decomposable bundle. In order to show this, we start by
remarking that the Abel-Jacobi map defined in (\ref{abeljaco})
is surjective, generically one-to-one and its fibers have positive dimension exactly over the $(g-2)$-dimensional subvariety of the Jacobian

\begin{equation}\label{symdue}
\sym^{g-2}_BC\eqdef \{L\in \Jac(C)| L\cong \OO(-B+ q_1 + \cdots + q_{g-2}),\ q_i\in C\}.
\end{equation}

In fact, for any line bundle $L\in \Jac(C)$, the fiber of the map
(\ref{abeljaco}) over $L$ has dimension $h^1(C,\cO_C(D-L))$ and
$h^1(C,\cO_C(D-L))>0$ if and only if $h^1(C,\cO_C(D-L))=h^0(C,\cO_C(K-D+L))=h^0(C,\cO_C(B+L))>0$,
i.e. when $L\in\sym^{g-2}_BC$. In the proof of the following
Proposition we use in an important way the assumptions of generality that
we make on the choice of $D\in \sym^g(C)$ and of $Q\in \sym^{g-2}(C)$.
In fact, particular choices of both divisors lead to geometric
configurations too intricate to be described in a reasonable space.

\begin{proposition}\label{exkum}
Let $E\in \overline{\varphi_B(\PP_B^{3g-6})}$ be a general semistable not stable bundle, then $\varphi_D^{-1}(E)\subset \Pd$ has two components: a rational family of $\PP^{g-1}$ and a $\PP^{g-1}$.
\end{proposition}

\begin{proof}
We remark that the general decomposable vector bundle contained in \\ $\overline{\varphi_B(\PP_B^{3g-6})}$ is of type $\OO(-B+Q)\oplus \OO(B-Q)$ for some effective degree $g-2$ divisor $Q\in \sym^{g-2}C$. In fact on the one hand we have

$$h^0(C,\cO_C(B)\otimes (\cO_C(Q-B)\oplus \cO_C(B-Q)))>0;$$

\noindent on the other hand, whenever $h^0(C,\cO_C(B)\otimes (L\oplus L^{-1}))\neq 0$ and $L\in \Jac(C)$, either $L$ or $L^{-1}$ is linearly equivalent to a line bundle of type $\cO_C(Q-B)$. The vector bundles of type $\OO(-B+Q)\oplus \OO(B-Q)$ are exactly the image in $\Kum(C)=\Jac(C)/\pm Id$ of the variety $\sym^{g-2}_BC \subset \Jac(C)$ via the usual quotient.

\smallskip

Let us describe the closure of the fiber of $\varphi_D$ over
$\OO(-B+Q)\oplus \OO(B-Q)$, when $Q$ is a general element of
$\sym^{g-2}C$. We claim that the fiber of $\varphi_D$ over the
$s$-class of the generic $\OO(-B+Q)\oplus \OO(B-Q)$ is given by two
components: the 1-dimensional rational family of $\PP^{g-1}$ spanned
in $\Pd$ by degree $g$ divisors of $|K-Q|\cong \PP^1$ plus a
$\PP^{g-1}$ spanned in $\Pd$ by the only effective divisor in
$|D-B+Q|$. This comes from Proposition \ref{span}, applied exactly as in Proposition \ref{piani}: if a vector bundle is of the form $L\oplus L^{-1}$, for $L\in \Jac(C)$, then the fiber of $\varphi_D$ is the union of the spans in $\Pd$ of the effective divisors $F$ such that $D-F\equiv L$ or $L^{-1}$. The only difference is that now we are looking at the locus where the map $\sigma_g$ of (\ref{abeljaco}) has positive dimensional fibers.\\

First, let us then describe the fibers of the map (\ref{abeljaco}) over
line bundles $\cO_C(Q-B)$ and $\cO_C(B-Q)$. If $L=\cO_C(Q-B)$ then $\cO_C(D-L)=\cO_C(K-Q)$, hence
$h^0(C,\cO_C(D-L))=2$ by the geometric Riemann-Roch Theorem. If $L=\cO_C(B-Q)$ we have $h^0(C,\cO_C(D-L))=h^0(C,\cO_C(D-B+Q))=1$ for a general $Q\in \sym^{g-2}(C)$. In fact $h^0(C,\cO_C(D-B+Q))=h^1(C,\cO_C(D-B+Q)+1=h^0(C,\cO_C(2B-Q))+1$ and we have $h^0(C,\cO_C(2B))=g-3$ if $D$ is general enough. This implies that for a general $Q$, we have $h^0(C,\cO_C(D-B+Q))=1$.

\smallskip

Now, observe that the above effective degree $g$ divisors in
the fibers of (\ref{abeljaco}) span $(g-1)$-dimensional projective linear spaces. In fact, for any set of points $P_1,...,P_g$ lying on $C$ embedded by the linear system $|K+2D|$, we have $h^0(C,K+2D-P_1...-P_g)=2g-1$, hence the linear span of such
points in $\Pd$ is a $\PP^{g-1}$.

\smallskip

By applying Prop. \ref{span}, exactly in the same way as we
did in the proof of Prop. \ref{piani}, we see that all these linear
spans are sent to the $s$-class of $\OO(-B+Q)\oplus \OO(B-Q)$ via
$\varphi_D$ and that this is the whole fiber of $\OO(-B+Q)\oplus
\OO(B-Q)$.

\end{proof}

\section{A projection in $\su$ and its determinantal interpretation}\label{cuore}

As we have seen in the preceding section, the linear span of $\varphi_B(\Pb)$ is a linear subspace of $|2\Theta|$ of projective dimension $\left[\sum_{i=0}^{g-2}\left(\begin{array}{c}
g \\ i\\
\end{array}\right)\right]-1$, from now on it will be denoted by $\PP_c$. It is not difficult to see that any linear subspace complementary (i.e. disjoint and with maximal dimension) to $\PP_c$ in $|2\Theta|$ has projective dimension $g$.

\smallskip
Let $p_{\PP_c}$ be the linear projection in $|2\Theta|$ with center $\PP_c$. Its linear target space then is a $g$-dimensional projective space.

Before stating the next proposition we recall that if $E\in \su - \overline{\varphi_B(\Pb)}$ is stable then we have $h^0(C,E\otimes \cO_C(D))=2$ and we denote by $s_1$ and $s_2$ a basis of $H^0(C,E\otimes \cO_C(D))$.

\newcommand\Pg{\widetilde{\PP}^g}

\begin{theorem}\label{determinant}
There exists a g-dimensional linear target projective subspace of $|2\Theta|$,
which can be identified with the linear system $|2D|$ on the curve
$C$, such that the restriction to $\su - (\Kum(C)\cup
\overline{\varphi_B(\Pb)})$ of the projection $p_{\PP_c}$ coincides
with the following determinant map

\begin{eqnarray}
\su-(\Kum(C) \cup \overline{\varphi_B(\Pb)}) & \lra & |2D|,\\
E & \mapsto & Zeros(s_1 \wedge s_2).\label{formula}
\end{eqnarray}
\end{theorem}

\begin{proof}

The strategy of the proof is to translate, via the theta map, the
description of $p_{\PP_c}$ from the language of vector bundles to
that of theta divisors. First (Step 1), we show that $\pic^{g-1}(C)$ contains a canonical model $\tilde{C}$ of $C$ such that its linear span in $|2\Theta|^*$ corresponds to the complete linear system $|2D|^*$ on $C$. This implies that $|2\Theta_{\tilde{C}}|\cong|2\Theta|_{\tilde{C}}\cong |2D|$. Then (Step 2) we show that the linear span of $\tilde{C}$ is the annihilator of $\PP_c$. The projection $p_{\PP_c|\su}$ determines then a hyperplane in the annihilator of $\PP_c$, i.e. a point in $|2\Theta|_{\tilde{C}}=|2D|$. We can then identify (Step 3) the target of $p_{\PP_c|\su}$ and $|2D|$.
Finally (Step 4) by an easy Riemann-Roch argument we show that this map is actually the one defined by (\ref{formula}).


\medskip

\textbf{Step 1.} Let $\Kum '(C)$ be the quotient of $\pic^{g-1}(C)$ via the involution $L\mapsto K - L$, for $L\in \pic^{g-1}(C)$. The Kummer variety $\Kum '(C)$ is naturally contained in $|2\Theta|^*$. 





\newcommand{\suk}{\mathcal{SU}_C(2,K)}

Now let us recall that there exists a "dual" moduli space
$\mathcal{SU}_C(2,K)\subset |2\Theta|^*$ of
semi-stable rank 2 vector bundles with canonical determinant. The moduli space $\suk$ is isomorphic to $\su$ and contains $\Kum '(C)$ as the locus of split bundles $L\oplus K-L$, for $L\in \pic^{g-1}(C)$. By \cite{opp} (Section 1), $\suk$ can be ruled by $g$-dimensional projective spaces $\PP H^0(C,2K-2W)^*$ parametrizing vector bundles that are written uniquely as extensions

$$0\lra W \lra E \lra K-W \lra 0,$$

for $W\in \pic^{g-2}(C)$. In particular, if we take $W=B$ we see that $\suk$ contains a $g$-dimensional projective space, that can be identified with $|2D|^*$. The intersection of this subspace with $\Kum '(C)$ is exactly the image of the curve $C$, embedded by the linear system $|2D|$, and its points correspond to split vector bundles of type $\cO_C(B+p)\oplus \cO_C(D-p),\ p\in C$. Let us denote by $\tilde{C}\subset \pic^{g-1}(C)$ the curve given by line bundles $\cO_C(B+p)$, for $p\in C$. This description implies that $|2\Theta|$ cuts out on $\tilde{C}$ the complete linear system $|2D|$.

\textbf{Step 2.} Recall how the theta embedding is defined: we have

\begin{eqnarray*}\label{theta}
\theta: \su & \lra & |2\Theta|,\\
E & \mapsto & \theta(E)=\{L\in \pic^{g-1}(C): h^0(C,E\otimes L)\neq
0\}.
\end{eqnarray*}

Now, $p_{\PP_c}$ is the linear projection with center $\PP_c$. The
hyperplanes of $|2\Theta|$ containing $\PP_c$, i.e. containing all
$\varphi_B(\Pb)$, correspond to points of $|2\Theta|^*$ belonging to
$\Xi:=\bigcap_{E\in \overline{\varphi_B(\Pb)}}\theta(E)$. Then $\Xi$ is the
locus of line bundles $F\in \pic^{g-1}(C)$ such that $h^0(C,E\otimes
F)\neq 0$ for any $E\in \overline{\varphi_B(\Pb)}$. The
linear span of $\Xi$ in $|2\Theta|^*$ has projective dimension $g$ and it is
the annihilator of $\PP_c$. 

We claim that $\tilde{C}$ is contained in $\Xi$: for any $p\in C$, $\cO_C(B+p)\in \Xi$. In fact, if $E\in\overline{\varphi_B(\Pb)}$ then there exists an injective morphism $\cO_C(-B)\ra E$. Let us twist the morphism by $F\in \pic^{g-1}(C)$ and let us take cohomology. Thus we have an injection

$$H^0(C, F\otimes \cO_C(-B))\hookrightarrow H^0(C,F\otimes E)$$

\noindent and we see that if $F\equiv B+p$ for some $p\in C$, then
$h^0(C,F\otimes \cO_C(-B))\neq 0$. Hence $h^0(C, F\otimes E)\neq 0$
and $F\in \Xi$. We strongly suspect that $\tilde{C} = \Xi$ but we
are not able to show this. By Step 1, the linear span of $\tilde{C}$ in $|2\Theta|^*$ is $|2D|^*$ and has projective dimension $g$. Since $\tilde{C}\subset \Xi$, this means that $\tilde{C}$ and $\Xi$ have the same linear span in $|2\Theta|^*$. 

\smallskip

\textbf{Step 3.} Now let us describe our projection. For a stable $E\in \su \subset |2\Theta|$, the linear space $<E,p_{\PP_c}>$ is the
intersection of all the hyperplanes of $|2\Theta|$ containing
$p_{\PP_c}$ and $E$. By Step $2$ we have that these hyperplanes
correspond to points of $|2\Theta|^*$ belonging to the intersection
of $\theta(E)$ with $\Xi$. Hence, we can naturally describe
$p_{\PP_c}$ in terms of theta divisors as the map that sends $E\in
\su$ to the restriction $\theta(E)_{|\Xi}$. 

\begin{eqnarray*}
\su-(\Kum(C) \cup \PP_c) &  \lra & |2\Theta|_{|\Xi},\\
E & \mapsto & \Delta' (E).
\end{eqnarray*}

\noindent where we set

\begin{equation}\label{defdelta}
\Delta' (E)\eqdef\{L\in \Xi|h^0(C,E\otimes L)\neq 0\}.
\end{equation}

By Step $2$ we know that $\Xi$ has the same linear span as $\tilde{C}$ in
$|2\Theta|^*$. This means that the natural restriction map

$$|2\Theta|_{\Xi} \lra |2\Theta|_{\tilde{C}}\cong |2D|$$

\noindent is an isomorphism and $p_{\PP_c}$ can be alternatively defined as the map sending a vector bundle $E$ to the divisor $\Delta(E)$, restriction of $\theta(E)$ to $\tilde{C}$, such that

\begin{equation}\label{defdelta2}
\Delta(E) \eqdef \{p\in C|h^0(E\otimes \cO_C(B+p))\neq 0\}.
\end{equation}

Let us then consider the map $p_{\PP_c}$ in the following form

\begin{eqnarray}
\su-(\Kum(C) \cup \PP_c) &  \lra & |2D|\nonumber\\
E & \mapsto & \Delta(E).\nonumber
\end{eqnarray}

\textbf{Step 4.}  In order to show that our projection coincides with the
determinant map (\ref{formula}) we follow the lines of \cite{bol:kumwed} (Lemma
1.2.3) that we recall for convenience of the reader. Let $p\in C$,
if $p\in Zeroes(s_1\wedge s_2)$ then there exists $s_p\in
H^0(C,E\otimes \cO_C(D-p))$ and thus $h^0(C,E\otimes\cO_C(D-p))\neq
0$. Now via Riemann-Roch and Serre-Duality (recalling that $E\cong
E^*$ and $B\equiv K-D$) one gets
$h^0(C,E\otimes\cO_C(D-p))=h^0(C,E\otimes \cO_C (B+ p))$. We recall
that $\Xi$ contains all the points $\cO_C(B+P)\in \pic^{g-1}(C)$.
This implies that, when $s_1,s_2\in H^0(C,E(D))$, the divisor of
zeroes of $s_1\wedge s_2$ is $\Delta(E)$ and the two maps coincide.

\end{proof}

We warn the reader that we will often abuse notation by denoting $N$ both the point of $|2D|$ and the set of points of the divisor $N$ on $C$ itself.\\

Recall that there is a projective model of $C\subset \Pd$, embedded by $|K+2D|$. Let us consider now the linear subspace $<N> \subset \Pd$ generated by the $2g$ points of a divisor $N\in |2D|$. We remark that the annihilator vector space of $<N>$ is $H^0(C,K + 2D - N)$, which has linear dimension equal to $g$. This means that the linear span $<N>\subset\Pd$ is a $\PP^{2g-2}$, and we shall denote it by $\PP^{2g-2}_N$.

The following Lemma comes directly from \cite{ln85} Prop. 1.1.

\begin{lemma}\label{lanna2}
Let $N\in |2D|$ and let $e \in \Pd$ be an extension

$$0\lra \OO(-D)\stackrel{i_e}{\lra} E_e \stackrel{\pi_e}{\lra}\OO(D)\lra 0.$$

Then $e \in \PP^{2g-2}_N$ if and only if there exists a section\\
$\alpha \in H^0(C,Hom(\OO(-D),E))$ such that $Zeroes(\pi_e \circ
\alpha)=N.$
\end{lemma}


In the next lemma we go through the relation between the fibers of the projection
$p_{\PP_c}$  and the classifying map $\varphi_D$.

\begin{proposition}\label{fibre}

Let $N\in |2D|$ be a general divisor on $C\subset \Pd$ and $\PP^{2g-2}_N\subset \PP^{3g-2}_D$ the linear span of the points of $N$.
Then the image of

$$\varphi_{D|\PP^{2g-2}_N}:\PP^{2g-2}_N\dashrightarrow \su$$

\noindent is the closure in $\su$ of the fiber over $N\in |2D|$ of the projection $p_{\PP_c}$.

\end{proposition}

\begin{proof}


Let $e\in\PP^{2g-2}_N$ and let $E_e=\varphi_D(e)$ be
its image in $\su$. Then, by
Lemma \ref{lanna2}, the extension class $e$ belongs to
$\PP^{2g-2}_N$ if and only if there exists a section $\alpha \in
H^0(C,Hom(\cO_C(-D),E_e))$ such that, in the notation of Lemma
\ref{lanna2}, we have that $Zeroes(\pi_e \circ \alpha)=N$. This in
turn implies that $\alpha$ and $i_e$ are 2 independent sections of
$E_e\otimes \cO_C(D)$ and that $Zeroes(i_e\wedge\alpha)=N$. Hence
Theorem \ref{determinant} implies that $E_e$ is projected on $N\in
|2D|$ and that the image of
$\varphi_{D|\PP^{2g-2}_N}$ is contained in the closure of the fiber of $p_{\PP_c}$ over $N$. 

The same argument, in the opposite sense, implies that,
for any stable bundle $E$ not contained in $\PP_c$, the one dimensional fiber
 $\varphi_D^{-1}(E)$ is contained in $\PP^{2g-2}_{\Delta(E)}$, where
 $\Delta(E)\in |2N|$ is the divisor associated to a vector bundle $E$, as defined in Equation (\ref{defdelta2}). This completes the proof.

\end{proof}

In other words, let $<N,\PP_c>$ denote the linear span in $|2\Theta|^*$ of $\PP_c$ and the point of $|2D|$ corresponding to $N$. Then we have that

$$\su \cap <N,\PP_c> = \overline{\varphi_B({\Pb})} \cup \overline{\varphi_D(\PP^{2g-2}_N)}.$$

\section{Generic fibers, rational normal curves and pointed genus 0 curves}

Now let us consider the fiber of $\varphi_D$ over a general bundle.
General, for what matters to us, will mean belonging neither to
$\Kum(C)$ nor to $\overline{\varphi_B(\Pb)}$. Let us define

\begin{equation}\label{sezia}
\Sec^N\eqdef\Sec^{g-1}(C)\cap \PP^{2g-2}_N
\end{equation}

\noindent for any generic $N\in |2D|$. Moreover we shall denote by
$\Sec^n(N)$ the configuration of $(n-1)$-linear spaces spanned in
$\PP^{2g-2}_N$ by $n$-ples of points of $N$. It would be natural to
expect that $\Sec^N=\Sec^{g-1}(N)$ but this does not always hold. As
the next Lemma shows, this depends on the genus of the curve $C$. On the other hand the inclusion $\Sec^{g-1}(N) \subseteq \Sec^N$ is always true, since all the points of $N$ are contained in $\PP^{2g-2}_N$.

\begin{lemma}\label{baselocus}
i) The restriction map

$$\varphi_{D|\PP^{2g-2}_N}:\PP^{2g-2}_N\dashrightarrow \su$$

\noindent is given by a linear sub-system of $|\cI_{\Sec^N}(g)|$.\\

\smallskip

ii) If $g<4$ then $\Sec^N=\Sec^{g-1}(N)$, otherwise
$\Sec^{g-1}(N)\subsetneq \Sec^N$, i.e. $\Sec^{g-1}(N)$ is strictly contained in the base locus of $\varphi_{D|\PP_N^{2g-2}}$.

\end{lemma}

In the following proof we will occasionally go back to the notation $|K+2D|^*$ for $\Pd$
since it seems easier to handle while considering annihilators and linear spans.

\begin{proof}

i) By Lemma \ref{linear}, Lemma \ref{luoghi} and Equation (\ref{bertro}) this is straightforward.


\smallskip

ii) First note that $C\cap \PP^{2g-2}_N=N$. In fact, since for every
$c\in C$ we have $h^0(C,K+2D-N-c)=h^0(C,K - c)=g-1<g=h^0(C,K+2D-N)$,
there cannot be any further intersection. Hence $\Sec^{g-1}(N)\subseteq \Sec^N$. Now, we need to show that if $g\geq 4$ there exists some $(g-2)$-dimensional
$(g-1)$-secant plane of $C$ in $\Pd$ intersecting $\PP^{2g-2}_N$ out
of $\Sec^{g-1}(N)$. In order to do this it is enough to find an
effective divisor on $C$ of degree $g-1$, not contained in $N$, such that its linear span has non empty intersection with $\PP^{2g-2}_N$. Recall that
$\PP^{2g-2}_N\subset \Pd$ is the annihilator of $|K+2D-N|=|K|\subset
|K+2D|$; hence there are $g$ sections of $H^0(C,K+2D)$, seen as
hyperplanes on $|K+2D|^*$, that vanish on $\PP^{2g-2}_N\subset
|K+2D|^*$. On the other hand, given any effective degree $(g-1)$
divisor $L_{g-1}$, via Riemann-Roch Theorem one sees that the
annihilator of the linear span $<L_{g-1}>\subset |K+2D|$ is given by
$2g$ sections of $H^0(C,K+2D)$ (i.e. $2g$ hyperplanes in
$|K+2D|^*$). Hence, since $h^0(C,K+2D)=3g-1$ we see that
$\PP^{2g-2}_N$ has non-empty intersection with $<L_{g-1}>$ if and
only if $\dim(H^0(C,K)\cap H^0(C,K+2D-L_{g-1}))\geq 2$ (note that
this means exactly the condition we want to check, i.e.
$\Sec^{g-1}(N)\subsetneq \Sec^N$). This in turn means that in
$\PP^{g-1}=|K|^*$ there exists a linear subspace of codimension at
least 2 that contains the points of $L_{g-1}$, i.e.
$h^0(C,K-L_{g-1})=h^1(C,L_{g-1})\geq 2$. By the geometric form of
Riemann-Roch Theorem, this is equivalent to

$$\dim(|L_{g-1}|)\geq g-1-1-(g-3)=1.$$

Hence, we have $\Sec^{g-1}(N)\subsetneq \Sec^N$ as long as $\dim(|L_{g-1}|)\geq 1$.
Finally, by the Existence Theorem of Brill-Noether theory (Thm. 1.1, page 206, \cite{acgh:gac})
we see that this is the case if and only if $g\geq 4$ (remember that we assume that $C$ is non-hyperelliptic). Furthermore the dimension of the variety $G^r_{g-1}$ of the linear series $g^r_{g-1}$ for $C$ is $g-(r+1)^2$, so for $r=1$ it has dimension $g-4$.
\end{proof}

We also need a slight generalization of a classical lemma. We will call \textit{degenerate} or \textit{singular} RNC a connected curve of degree $d$ and genus 0 in $\PP^d$ that is the union of smooth RNC of degree smaller than $d$.

\begin{lemma}\label{castelnuovo}
Let $p_1,\dots, p_{n+2}$ a set of linearly general points in $\PP^n$. Given any further point $q\in \PP^n$, there exists at least one (possibly singular) RNC through the $(n+3)$ points $p_1,\dots,p_{n+2},q$.
\end{lemma}

\begin{proof}
If the set of $(n+2)$ points $p_1,\dots,p_{n+2},q$ is general, then the result is classical (see \cite{har:ag}, Thm. 1.18). Suppose instead that $q$ lies in some proper linear subspace $\PP^m \subset \PP^n$, $m<n$, spanned by a subset of the $p_i$s. Up to relabeling we can assume $\hat{\PP}^m:=\langle p_1,\dots,p_{m+1} \rangle$. The case $q=p_i$ for some $i$ is clear. By possibly taking a smaller linear span we can assume that $q$ is in general position with respect to $p_1,\dots,p_{m+1}$. Let us denote by $t$ the intersection point of $\hat{\PP}^m$ with $\hat{\PP}^{n-m}:=\langle p_{m+2},\dots,p_{n+2}\rangle$. Then by the hypotheses of generality on the $p_i$s, the $(m+3)$ points $p_1,\dots,p_{m+1},q,t$ are in general position inside $\PP^m$, hence by the classical case there exists a unique RNC $C_m$ of degree $m$ passing through $p_1,\dots,p_{m+1},q,t$. On the other hand, the $(n-m+2)$ points $p_{m+2},\dots,p_{n+2},t$ in $\hat{\PP}^{n-m}$ are in general position as well, hence by the classical case there exists at least one  RNC $C_{n-m}$ of degree $(n-m)$ passing through them. The nodal reducible RNC that we want to display is just the union $C_m\cup C_{n-m}$.
\end{proof}

Lemma \ref{baselocus} has some unpleasant consequences for the
description of the fibers of $p_{\PP_c}$ for $g>3$ (see Remark
\ref{negative}). In the following Proposition, we will consider the
case in which $N\in |2D|$ is the divisor $\Delta(E)$ associated to a
vector bundle $E$, as defined in Equation
(\ref{defdelta2}). The points of $\Delta(E)$ are in general position in $\PP_{\Delta(E)}^{2g-2}$.


Let us denote by $C_E$ the projective closure of the fiber
$\varphi^{-1}_D(E)$. In the following we will often write RNC instead of
Rational Normal Curve. Recall that in $\PP^m$ there exists a unique RNC passing through $m+3$ points in general position.

\begin{proposition}\label{rnc}
Let $E$ be a general vector bundle in $\su -(\Kum(C) \cup \overline{\varphi_B(\Pb)})$. Then $C_E\subset \PP^{2g-2}_{\Delta(E)}$ is a rational normal curve of degree $2g-2$ passing through the $2g$ points of $\Delta(E)\subset \PP^{2g-2}_{\Delta(E)}$.
\end{proposition}

\begin{remark}\label{quantogeneral}
One can be slightly more precise about the hypotheses on $E$. By "general'' inside $\su -(\Kum(C) \cup \overline{\varphi_B(\Pb)})$, we mean first of all that $\Delta(E)$ is reduced. Moreover, we also exclude the vector bundles $E$ such that there exists a divisor $G \subset \Delta(E)$, with $deg(G)<2g-1$ and
an extension class $(f)\in \varphi_D^{-1}(E)$ that has a sheaf injection $\OO(D - G)\hookrightarrow E$ vanishing on $G$. The fibers over the stable bundles for which such a divisor $G$ exists are lower degree RNCs contained in linear subspaces (in fact spanned by $G$) of $\PP^{2g-2}_{\Delta(E)}$.
(see Prop. \ref{span}). 
\end{remark}

\begin{proof}
The strategy of the proof is the following. First we describe the base locus of $\varphi_{D|\PP^{2g-2}_{\Delta(E)}}$, then we show that the generic fiber $C_E$ must be a finite set of RNCs and finally, by a birationality argument, we show that it must consist of only one RNC.\\
\smallskip

We recall that $\su -(\Kum(C) \cup \overline{\varphi_B(\Pb)})$ is the locus where $C_E$ has dimension one. By Proposition \ref{fibre} we know that $C_E\subset\PP^{2g-2}_{\Delta(E)}$. Moreover, we have the trivial numerical equality

\begin{equation}\label{coincidenza}
(2g-2)g=(g-1)2g
\end{equation}

\noindent that has some useful implications. We know from Lemma \ref{baselocus} (ii) that the base locus of $\varphi_{D|\PP_{\Delta(E)}^{2g-2}}$ contains strictly $\Sec^{g-1}(\Delta(E))$ if $g>3$. We claim that in this case the further base locus of $\varphi_{D|\PP_{\Delta(E)}^{2g-2}}$ is set-theoretically a family of rational normal curves passing through the $2g$ points of $\Delta(E)$. In fact, recall that $\varphi_D$ is defined by forms of degree $g$ vanishing with multiplicity $(g-1)$ at the points of $C$. Hence, by Bezout Theorem, Lemma \ref{baselocus} and Equation (\ref{coincidenza}) if all forms that define $\varphi_{D|\PP^{2g-2}_{\Delta(E)}}$ vanish at least on one further point $p$ with respect to $Sec^{g-1}(\Delta(E))$, then they are forced to vanish on the possibly degenerate RNCs through $\Delta(E)$ and $p$ (see Lemma \ref{castelnuovo}), otherwise Bezout Theorem would be contradicted. If $p$ is general there will be just one RNC, otherwise there may be more (Lemma \ref{castelnuovo}).
Note however that $\Sec^{\Delta(E)}$ (with the notation of equality \ref{sezia}) has codimension at least two in $\PP^{2g-2}_{\Delta(E)}$ since $\dim(\Sec^{g-1}(C))=2g-3$ and $\Sec^{g-1}(C)$ cannot be a proper subset of $\PP^{2g-2}_{\Delta(E)}$. Since the base locus for $\varphi_{D|\PP^{2g-2}_{\Delta(E)}}$ is exactly $\Sec^{\Delta(E)}$ and $\PP^{2g-2}_{\Delta(E)}$ is ruled by RNCs through $\Delta(E)$ (see Lemma \ref{castelnuovo}), the generic RNC through $\Delta(E)$ is not contained in the base locus.

Furthermore the equality (\ref{coincidenza}) implies that the classifying map $\varphi_D$ is constant along the RNCs (not contained in its base locus) passing through the $2g$ points of $\Delta(E)$. In fact by Lemma \ref{baselocus} the restriction $\varphi_{D|\PP^{2g-2}_{\Delta(E)}}$ is given by a linear subsystem of  $|\cI_{\Delta(E)}^{g-1}(g)|$ (if $g=2,3$ it is given by the full linear system) thus the zero loci of the forms of this linear system can not have intersection bigger than $\Delta(E)$ with our RNCs.
This means that $C_E$ is a finite collection of RNCs passing through $\Delta(E)$. Note moreover that $\Delta(E)$ is the only intersection of $C_E$ with the base locus of $\varphi_{D|\PP^{2g-2}_{\Delta(E)}}$. As we have already stated, if $C_E$ has  further intersection with the base locus, then it is contained in the base locus itself.

\smallskip

In order to prove that $C_E$ is in fact a unique curve, let us recall that every point $e\in \varphi_D^{-1}(E)$ represents an exact sequence like the following

\begin{equation}\label{seq1}
0\ra \cO_C(-D) \ra E \ra \cO_C(D) \ra 0.
\end{equation}

\noindent Thus we can define a map

\begin{equation}\label{acca}
h:C_E \ra \PP H^0(C,E\otimes \cO_C(D)))=\PP^1
\end{equation}

\noindent that sends the extension class $e\in \varphi_D^{-1}(E)$ on the point $h(e)\in \PP H^0(C,E\otimes \cO_C (D))$ corresponding to the first morphism of the exact sequence (\ref{seq1}). The map $h$ is birational since on the open set complementary to $\Delta(E)$ it is one-to-one. This in turn implies that $C_E$ must be just one irreducible RNC. In fact, the arithmetic genus is forced to be 0. Note that, \textit{a priori}, the map $h$ from Equation (\ref{acca}) is not defined on $\Delta(E)=C_E\cap C$ but one can complete it by sending any $p\in \Delta(E)$ to the only section $s_p$ of $E\otimes \cO_C(D)$ that vanishes on $p$ (see Proof of Thm. \ref{determinant} for the definition of $s_p$).


\end{proof}



The linear systems contracting all RNC passing through a set of fixed points have been explored in a detailed way in \cite{bolgit}.
The space of rational normal curves in $\PP^{m-2}$ passing through $m$ points in general position is closely related to the moduli space $\cM_{0,m}$ of configurations of ordered distinct $m$ points on $\PP^1$, as the following Theorem shows. By $\cH$ we denote the Hilbert scheme classifying flat closed finitely presented subschemes of $\PP^{m-2}$. It is a disjoint union indexed by the various Hilbert polynomials.

\begin{theorem}\label{verokapra}(\cite{kapra}, Thm. 0.1))
Take m points $q_1,\dots,q_m$ in general position in the projective space $\PP^{m-2}$. Let $V_0(q_1,\dots,q_m)$ be the space of all rational normal curves in $\PP^{m-2}$ through the points $q_i$. Considering it as a subvariety of the Hilbert scheme $\cH$ parametrizing all subschemes of $\PP^{m-2}$, we have $V_0(q_1,\dots,q_m)\cong \cM_{0,m}$.
\end{theorem}

Moreover Kapranov showed that, if we take the closure

$$V(q_1,\dots,q_m):=\overline{V_0(q_1,\dots,q_m)}$$

\noindent of $V_0(q_1,\dots,q_m)$ in $\cH$ then we get $\overline{\cM}_{0,m}$, i.e. the compactification of $\cM_{0,m}$ obtained by adding stable curves.

\begin{remark}\label{zeito}
When $g(C)=2$ and $D=K$ Proposition \ref{rnc} coincides with the description of the conic bundle given in \cite{bol:kumwed}. In this case the closure of the fiber of the classifying map over a stable bundle $E$ is a plane conic in $\PP^2_{\Delta(E)}$ passing through the four points of $\Delta(E)\in|2K|$. The map $p_{\PP_c}$ is the projection on $|2K|=\PP^2$ with center the node $[\cO_C\oplus \cO_C]$ of the Kummer surface $\Kum(C)\subset |2\Theta|$ and the fiber of $p_{\PP_c}$ over a divisor $\Delta(E)\in |2K|$ is a $\PP^1$, that corresponds to the pencil of conics in $\PP^{2}_{\Delta(E)}$ passing through the four points. This pencil can be seen as the base example of Theorem \ref{verokapra}. Plane conics passing through four fixed points in general position are in fact in bijection with configurations of four points on the projective line and $\PP^1$ is in fact the GIT compactification of $\cM_{0,4}$ (and, by the way, also the Mumford-Knudsen one, $\overline{\cM}_{0,4}$). The semistable configurations correspond to the rank 2 reducible conics and to the points of intersection of the projective line with $\Kum(C)$.
\end{remark}

Proposition \ref{rnc} and Theorem \ref{verokapra} then suggest us a modular interpretation of the fibers of the projection $p_{\PP_c}$. In fact we have just showed that, if $N\in |2D|$, then the restriction

$$\varphi_{D|\PP^{2g-2}_N}:\PP^{2g-2}_N\dashrightarrow \su$$

\noindent contracts every RNC passing through the $2g$ points of N, when the RNC is not contained in the base locus. In particular, if $g=2,3$ then $\varphi_{D|\PP^{2g-2}_N}$ contracts every RNC passing through $N$. Then, via Kapranov's isomorphism, the fibers of the map $\varphi_D$ make up a family of $2g$-pointed rational curves. The last two sections of the paper will be devoted to describing the fibers of $p_{\PP_c}$ via the families of RNCs contained in $\Pd$.

\section{The genus 3 case: a fibration in Segre cubics.}

Let us now go through the details of the genus 3 case, assuming that $C$ is not hyperelliptic. The general genus $g$ case will be developed in the next Section. As already stated, in this case $\su$ is embedded in $\PP^7=|2\Theta|$ as a quartic hypersurface singular along  $\Kum(C)$, first discovered by Coble \cite{cob:agtf}.

In this case $deg(D)=3$ and $\varphi_B$ is a linear embedding of $\PP^3_B$ in $\PP^7$ (see \cite{paulydual} Sec. 2.3). The image of the projection from $\PP_c=\PP^3_B$ is a $\PP^3$ as well, that is identified with $|2D|$ by Theorem \ref{determinant}. On the other hand the extension classes belonging to $\ext^1(\cO_C(D),\cO_C(-D))$ are parametrized by a $\PP^7_D$ that contains a model of $C$ and the classifying map $\varphi_D$ is given by the complete linear system $|\cI_C^2(3)|$.


\begin{remark}
The choice of a projective model of $C$ in this case allows to do explicit calculations on this map, since they are still fairly simple and can be performed in a reasonable time by a computer. By computing the image of this map with \textit{Macaulay}, we found some equations of Coble quartics in terms of the coefficients of a plane quartic model of $C$. The same results, with methods coming from the context of integrable systems, were obtained by P. Vanhaecke in \cite{pol}.
\end{remark}

Let us now take a generic divisor $N\in |2D|$ and denote by $S_N$ the closure of its fiber $p_{\PP_c}^{-1}(N)$. For simplicity we will assume that $N$ is reduced: all points are distinct. By Proposition \ref{fibre} $S_N$ is the image via $\varphi_D$ of the $\PP^4_N$ spanned by the six points of $N$.

\begin{proposition}\label{segre}
The 3-fold $S_N$ is a Segre cubic.
\end{proposition}

The Segre cubic $S_3$ is a classical modular threefold (see for instance \cite{do:pstf}). In $\PP^5$ with homogeneous coordinates $[x_0 :\dots: x_5]$ we consider the complete intersection

$$S_3 := \bigl\{ \sum_{i=0}^5 x_i =0;
\sum_{i=0}^5 x^3_i= 0 \bigr\}.$$

The first equation is linear, so $S_3$ is a hypersurface in the $\PP^4:= \{ x\in \PP^5| \sum x_i = 0 \}$.
Using $[x_0 :\dots: x_5]$ as projective coordinates, the relation $x_5=-x_0 \dots -x_4$ gives the equation of $S_3$ as a hypersurface but the equation in $\PP^5$ has the advantage of showing that $S_3$ is invariant under the symmetry group $\Sigma_6$, acting on
$\PP^5$ by permuting coordinates, which is not immediate from the hypersurface equation. $S_3$ is the GIT compactification of the moduli space $\cM_{0,6}$ of ordered configurations of 6 points on $\PP^1$. By considering these points as Weierstrass points of a genus 2 curve, one can also see $S_3$ as a birational model of the Satake compactification of $\mathcal{A}_2(2)$, the moduli space of principally polarized abelian surfaces with a level 2 structure. In fact $S_3$ is the dual variety of the Igusa quartic (for an account see \cite{do:pstf} or \cite{hu:gsq}), that is the compactification of $\mathcal{A}_2(2)$ given by fourth powers of theta constants.

\begin{proof}\textit{(of Proposition \ref{segre})}

By the general description of the fibers over semistable non-stable bundles given in Prop. \ref{piani}, there are pairs of $\PP^2$s, spanned by complementary triples of points of $N\subset \PP^4_Q$, that are contracted to points of $Sing(\su)=\Kum(C)$. More precisely we have
$10=\left(\begin{array}{c}6\\ 3\\ \end{array}\right)/2$ pairs of $\PP^2$s that are contracted, each pair to a point of the intersection $S_N\cap \Kum(C)$, which is made up of ten points. Note that the intersection of $\su$ with the $\PP^4$ spanned by $\PP^3_B=\PP_c$ and the point in $|2D|$ corresponding to $N$ is the union $S_N \cup \PP^3_B \subset \PP^4$. Since $deg(\su)=4$ this implies that $deg(S_N)=3$. By results of Varchenko \cite{varchenko}, a cubic 3-fold cannot have more than ten isolated double points and the Segre cubic is the unique (up to isomorphism) cubic 3-fold with ten nodes (see also \cite{hu:gsq}, Sect. 3.2). Hence we conclude.
\end{proof}

\begin{remark}\label{negative}
Lemma \ref{baselocus} and Proposition \ref{rnc} imply that if $g>3$, we will not have a classifying map that contracts all RNCs, since some of them are contained in the base locus. Hence we cannot expect an isomorphism of the generic fiber with any compactification of $\cM_{0,2g}$.
\end{remark}


\smallskip





We remark that one can explicitly display set of \textit{boundary} divisors inside $S_N$ over which the fibers of $\varphi_D$ degenerate to lower degree rational normal curves. This degeneration of the fibers gives $\varphi_D$ a quite deep modular description in terms of $\cM_{0,6}$, that has been explored and generalized to other higher dimensional moduli spaces in \cite{bolgit}.


\medskip

Let us denote by $Bl_{\PP_c}(\su) \subset Bl_{\PP_c}|2\Theta|$ the blow-up of the Coble quartic along $\PP_c$, naturally contained in the blow-up of $|2\Theta|$ along the same subvariety. Since $\PP_c$ is a hyperplane of $<\PP_c,N>=\PP^4$ for every $N\in |2D|$, the blow up of $<\PP_c,N>$ is isomorphic to $\PP^4$ itself. Hence we have the following corollary.

\begin{corollary}
The blow up $Bl_{\PP_c}|2\Theta|$ is a rank 4 projective bundle over $|2D|$ that contains $Bl_{\PP_c}(\su)$. The intersection of $Bl_{\PP_c}(\su)$ with the general fiber of the projective bundle is a Segre cubic.
\end{corollary}

\section{The fibration in $\cM_{0,2g}$ for $g(C)\geq 4$.}

\subsection{The moduli spaces $\cM_{0,n}$}

Before going through the main results of this section we give a brief account of a few results on the moduli spaces of pointed curves that will be needed in the following. If $n\geq 3$, $\cM_{0,n}$ is set-theoretically the set of projective equivalence classes of ordered $n$-tuples of distinct points on $\PP^1$. Moreover, for the same range of $n$, it carries a structure of quasi-projective algebraic variety. There exist different compactifications of these spaces. The oldest one is probably the GIT compactification $\cM_{0,n}^{GIT}$ (the Segre cubic is $\cM_{0,6}^{GIT}$), that has known further recent interest thanks to \cite{vakilpointed}, where the equations for $\cM_{0,n}^{GIT}$ are computed. Subsequently the Mumford-Knudsen compactification  $\overline{\cM}_{0,n}$ was introduced, which is obtained by adding stable nodal marked curves \cite{knudsenpointed} and is finer than the GIT one on the boundary. By this we mean that there exists a birational morphism $\overline{\cM}_{0,n} \ra \cM_{0,n}^{GIT}$ which contracts partially some boundary strata, but it is an isomorphism over the open subset $\cM_{0,n} \subset \overline{\cM}_{0,n}$. See \cite{bolgit} or \cite{langeavritz} for more details on this morphism.


\medskip

A large amount of results on $\overline{\cM}_{0,n}$ can be obtained by studying the geometry of RNCs in $\PP^{n-2}$ and the birational transformations of $\PP^{n-2}$: the first example of this interplay was Theorem \ref{verokapra}. The Mumford-Knudsen space $\overline{\cM}_{0,n+1}$ has also different realizations as a blow-up of the projective space. We are particularly interested in the following.

\begin{theorem}\label{bloua}(\cite{kapra},\cite{hassetpointed} Sect. 6.2)

The Mumford-Knudsen compacification $\overline{\cM}_{0,n+1}$ has the following realization as a sequence of blow-ups of $\PP^{n-2}$. Let $q_1,\dots,q_{n}$ be general points in $\PP^{n-2}$:\\

\smallskip

\noindent 1: blow up the points $q_1,\dots, q_{n}$;\\
2: blow up proper transforms of lines spanned by pairs of the points
$q_1,\dots, q_{n}$;\\
3: blow up proper transforms of 2-planes spanned by triples of the point; . . .\\
...\\
n-3: blow up proper transforms of (n-4)-planes spanned by (n-3)-tuples
of the points.
\end{theorem}

The idea of Theorem \ref{bloua} is that one can associate to a general point $q\in\PP^{n-2}$ the unique rational curve passing through $q,q_1,\dots,q_n$. The points $q_i$ determine the first $n$ markings and $q$ the $n+1^{th}$. This gives in fact an element of $\overline{\cM}_{0,n+1}$.  The blow-down map $b:\overline{\cM}_{0,n+1}\ra \PP^{n-2}$ associated to the construction of Theorem \ref{bloua} was first described, via a different sequence of blow-ups, by Kapranov \cite{kapra}. The map $b$ shows quite explicitly the relation of $\overline{\cM}_{0,n+1}$ with RNCs in $V_0(q_1,\dots,q_n)$. Let in fact $\pi:\overline{\cM}_{0,n+1}\ra \overline{\cM}_{0,n}$ be the forgetful morphism that drops the $(n+1)^{th}$ point and note that the fibers of $\pi$ are the universal curve over $\overline{\cM}_{0,n}$. Now, by \cite{keelkernan} Prop. 3.1., the images via $b$ of the fibers over points of ${\cM}_{0,n}\subset \overline{\cM}_{0,n}$ are the rational normal curves in $\PP^{n-2}$ passing through the $n$ general points $q_i$.

\subsection{Birational geometry of the fibers of $p_{\PP_c}$.}

Let $N= p_1+\dots + p_{2g}\in |2D|$ be a reduced divisor. By lemma \ref{baselocus}, if $g(C)\geq 4$ then the configuration of linear spaces $\Sec^{g-1}(N)$ is strictly contained in the base locus $\Sec^N$ of the restricted map $\varphi_{D|\PP^{2g-2}_N}$. The further base locus is a family of rational normal curves in $V_0(p_1,\dots,p_{2g})$. We claim that the generic RNC in $V_0(p_1,\dots,p_{2g})$ is not contained in $\Sec^N$.
We remark in fact that $\Sec^N$ has codimension at least two in $\PP^{2g-2}_N$ since $\dim(\Sec^{g-1}(C))=2g-3$ and $\Sec^{g-1}(C)$ cannot be contained in $\PP^{2g-2}_N$. This in particular implies that the generic RNC in $V_0(p_1,\dots,p_{2g})$ is not contained in the base locus.

\medskip

\begin{theorem}\label{functor}

Let $g\geq 4$. Then there exists a family of 2g-pointed rational curves over an open set of the fiber $p_{\PP_c}^{-1}(N)$
that induces a birational map to the moduli space $\cM_{0,2g}$.
\end{theorem}

\begin{proof}
Let us denote by $S_N$ the fiber $p_{\PP_c}^{-1}(N)$. The dimension of $S_N$ is $2g-3$, which is the same as the dimension of $\overline{\cM}_{0,2g}$. 
Recall that the general RNC, not contained in the base locus of $\varphi_{D|\PP^{2g-2}_N}$ and belonging to $V_0(p_1,\dots,p_{2g})$ is contracted by $\varphi_D$ to the S-equivalence class of a stable bundle in $p_{\PP_c}^{-1}(N)$. By Proposition \ref{rnc}, these stable bundles are an open subset $U$ contained in the intersection of $<N,\PP_c>$ with the stable part of $\su$. For the precise definition of $U$ see Rem. \ref{quantogeneral}, setting $N=\Delta(E)$. Now let us blow up $\PP^{2g-2}_N$ recursively, as described in Theorem \ref{bloua}, until we obtain $\overline{\cM}_{0,2g+1}$. Remember that this is accomplished by blowing up recursively the proper transforms of linear spans of points of $N$ up to the codimension two ones. 
Let us denote by $\sigma: \overline{\cM}_{0,2g+1} \ra \PP^{2g-2}_N$ the blow down map and call $\tilde{\varphi}_D:\overline{\cM}_{0,2g+1}\dashrightarrow S_N$ the map $\varphi_D \circ \sigma$.  This is still just a rational map because the part of the base locus which is ruled by RNCs has not been resolved (see the proof of Prop. \ref{rnc}), but we have managed to separate all the RNCs in $V_0(p_1,\dots,p_n)$. In fact by \cite{keelkernan} Prop. 3.1, the RNCs in $V_0(p_1,\dots,p_n)$ are pulled-back via $\sigma$ to the fibers of the universal curve over $\overline{\cM}_{0,2g}$ inside $\overline{\cM}_{0,2g+1}$. Now the fibers over vector bundles in $U$ are then in the regular locus of $\tilde{\varphi}_D$ since they do not intersect the base locus, which is a subfamily of the universal curve with base of smaller dimension. Hence they give us a sub-family, defined over $U$, of dimension $2g-3$ of the universal curve over $\overline{\cM}_{0,2g}$.
This in turn, by the universal property of the moduli space, induces an embedding $\nu$ of the open set $U\subset p_{\PP_c}^{-1}(N)$ in $\overline{\cM}_{0,2g}$ thus yielding a birational map between $S_N$ and $\cM_{0,2g}$. The situation is then summarized in the following commutative diagram.

\medskip

\ \ \ \ \ \ \ \  \ \ \ \ \ \ \ \ \ \ \ \ \ \ \ \ \ \ \ \ \ \ \xymatrix{ \overline{\cM}_{0,2g+1}\ar[d]^{\sigma} \ar@{-->}[dr]^{\tilde{\varphi}_D} \ar[r]^{\pi}  & \overline{\cM}_{0,2g}\\
\PP^{2g-2}_N \ar@{-->}[r]^{\varphi_D} & \ \ \ U \ \ \subset S_N \ar@{^{(}->}[u]_{\nu}}

\end{proof}

Now, the main Theorem (Theorem \ref{global}) of this paper is a combination of Remark \ref{zeito} (see also \cite{bol:kumwed}), Proposition \ref{segre} and Theorem \ref{functor}. A naif picture of the situation is given in Figure \ref{naive}.

\begin{figure}[!h]

\centering

\includegraphics[scale=0.45]{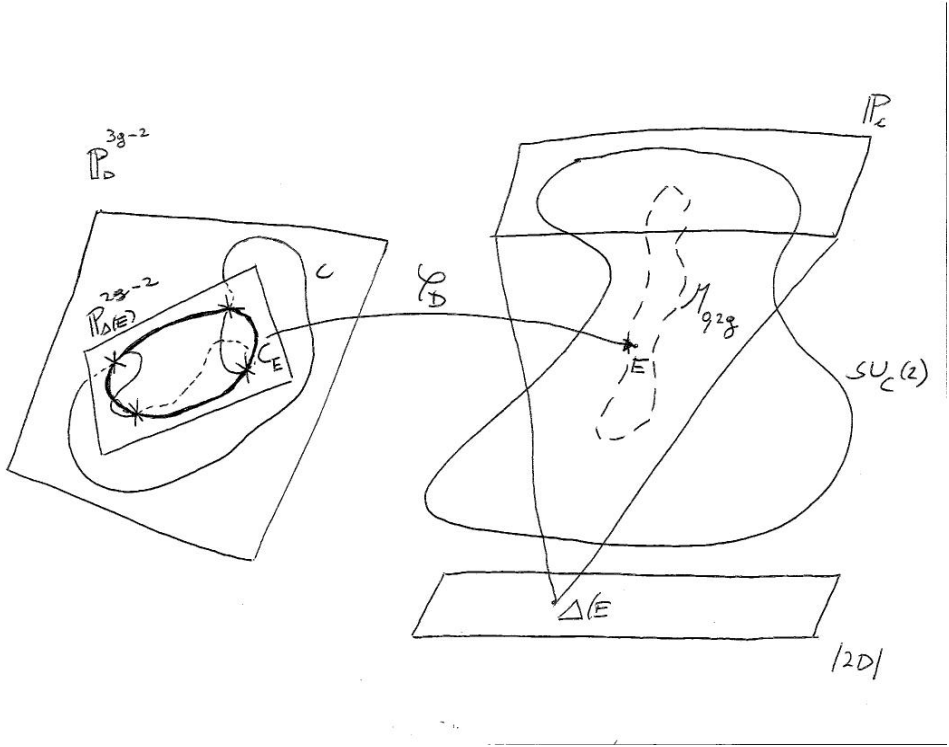}

\caption{A naif picture.\label{naive}}

\end{figure}

\medskip

It is worth reminding that there exists an explicit birational inverse map $\nu^{-1}:\overline{\cM}_{0,2g} \dashrightarrow p_{\PP_c}^{-1}(N)$ of $\nu$, where $N$ as usual is a reduced divisor $p_1+\cdots+p_{2g}\in |2D|$. We briefly describe it, being inspired by the map displayed in the proof of Thm 6.2 of \cite{bolobrivio}: we will denote by $\tilde{U}$ the image of $U$ inside $\overline{\cM}_{0,2g}$ via $\nu$. The idea is to build a semi-stable vector bundle with trivial determinant starting from a point configuration contained in $\tilde{U}$. Let us consider a vector space $V$ of dim 2, and set $\PP^1:=\PP(V^*)$. Let $(q_1,\dots,q_{2g})\in (\PP^1)^{2g}$ be an ordered $(2g)$-uple of distinct points represented by a moduli point in $\tilde{U}\subset\overline{\cM}_{0,2g}$. Let us choose lifts $(v_1,\dots,v_{2g})\in (V^*)^{2g}$ of the $q_i$s. More precisely, we associate to each $q_i$ an element $v_i$ of $Hom(V\otimes \OO, \mathcal{O}_{p_i})\cong V^*$. This in turn implies that we can produce a surjective morphism of sheaves $\kappa_{(q_1,\dots,q_{2g})}: V \otimes\OO \ra \mathcal{O}_N$ for each $(2g)$-uple of distinct points on $\PP^1$, by defining it to be zero out of $N$ and defined by $v_i$ on the fiber over $p_i$, for all $i$. The morphism depends on the choice of the lifts $v_i$ but it is not hard to see that the kernel

\begin{equation}\label{kern}
0\ra Ker(\kappa_{(q_1,\dots,q_{2g})}) \ra V \otimes \OO \ra \mathcal{O}_N \ra 0
\end{equation}

does only depend on the configuration of points $q_i\in \PP^1$. Moreover, the kernel is defined up to the choice of a basis of $V$, hence projectively equivalent ordered  point sets give rise to isomorphic kernels. This implies that we have a flat family of rank two bundles over $\tilde{U}$. Remark moreover that $\det(Ker(\kappa_{(q_1,\dots,q_{2g})}))=\OO(-2D)$. Then, results from Section 5 and 6 of \cite{bolobrivio} (in particular Thm. 6.2) imply the following Proposition.

\begin{proposition}
Let $(q_1,\dots,q_{2g})$ be any point configuration inside $\tilde{U}\subset 
\overline{\cM}_{0,2g}$. Then the map

\begin{eqnarray*}
\nu^{-1}:\tilde{U} & \ra & p_{\PP_c}^{-1}(N), \\
(q_1,\dots,q_{2g}) & \mapsto & Ker(\kappa_{q_1,\dots,q_{2g}})\otimes \OO(D);
\end{eqnarray*}

is a birational inverse of $\nu$.
\end{proposition}

In fact, under some assumptions one can extend $\nu^{-1}$ to some GIT semi-stable configurations of $2g$ points. Recall that a configuration of $2g$ points in $\PP^1$ is GIT semi-stable (resp. stable) when no more than $g$ points coincide (resp. no more than $(g-1)$).

\begin{proposition}\label{stablebun}
Let $(q_1,\dots,q_{2g})$ a GIT semi-stable configuration of points on $\PP^1$ and let us fix $N\in |2D|$ a reduced divisor. Then $K:=Ker(\kappa_{(q_1,\dots,q_{2g})})$ is stable if $h^0(C,K^*)<3$ and semi-stable if $h^0(C,K^*)<4$.
\end{proposition}

\begin{proof}
We will show the statement by proving semi-stability for the dual bundle $K^*$. By dualizing the exact sequence (\ref{kern}), we obtain:

$$0 \ra V^*\otimes \OO \ra K^* \ra \mathcal{O}_N \ra 0.$$

Passing to cohomology we see that $V^*$ injects in $H^0(C,K^*)$ as a 2-dimension subspace. Recall moreover that $\mu(K^*)=g$. Now, suppose that there exists a destabilizing line sub-bundle $R\subset K^*$ of degree $(g+1)$. By Riemann-Roch $h^0(C,R)\geq 2$. Hence if $h^0(C,K^*)<4$ there exists at least one common section
$s\in V^*\cap H^0(C,R) \subset H^0(C,K^*)$. Hence we get a commutative diagram

$$\xymatrix{   0  \ar[r] &  \OO  \ar@{^{(}->}[d]   \ar[r]^{\cdot s}  & R \ar@{^{(}->}[d]    \ar[r] & \mathcal{O}_F   \ar[d] \ar[r] & 0  \\
 0  \ar[r] & V^* \otimes \OO   \ar[r] & K^*  \ar[r]   &   \mathcal{O}_N \ar[r] & 0, \\
 }$$ 

where $F \subset N$ is a degree $(g+1)$ reduced divisor, and $\cdot s$ is the evaluation map of the section $s$. We observe that the bottom left map is just the evaluation of the sections in $V^*$. Furthermore, the lifts to $V^*$ of the points $(q_1,\dots,q_{2g})$ correspond to the locus inside $V^*\otimes \OO$ where the evaluation map $V^*\otimes \OO \ra K^*$ degenerates. Then one sees that this construction contradicts the hypotheses of semi-stability of $(q_1,\dots,q_{2g})$, since the degeneration locus of $\cdot s$ produces a set of $(g+1)$ coinciding points in $\PP^1=\PP(V^*)$ via the injection $\OO\hookrightarrow V^*\otimes \OO$. It is straightforward to check that this construction does not depend on the choice of the lifts of $(q_1,\dots,q_{2g})$. The arguments concerning stability or higher degree destabilizing bundles go along the same lines.
\end{proof}

\begin{remark}
We remark that, by twisting appropriately, this implies the semi-stability of $K\otimes \OO(D)$ hence one can define $\mu^{-1}$ over all the semi-stable configuration of points under the hypotheses of Prop. \ref{stablebun}.
\end{remark}

\begin{remark}
In \cite[Example 6.1]{bolobrivio} the authors display an example of a non stable rank 2 vector bundle, constructed via the same exact sequence (\ref{kern}) starting from a stable configuration of points. It is worth remarking that in fact in that example the vector bundle turns out to have \textit{four} global sections.
\end{remark}

\bibliographystyle{abbrv}
\bibliography{bibkaji}

\bigskip

Alberto Alzati\\
Dipartimento di Matematica "F.Enriques"\\
Via Saldini\\
20100 Milano\\
Italy.\\
alberto.alzati@unimi.it\\

\medskip

Michele Bolognesi\\
IRMAR\\
Universit\'{e} de Rennes1\\
263 Av. du G\'{e}n\'{e}ral Leclerc\\
35042  RENNES C\'{e}dex\\
FRANCE.\\
michele.bolognesi@univ-rennes1.fr

\end{document}